\documentclass[11pt]{amsart}
\usepackage{mathrsfs,latexsym,amsfonts,amssymb}
\newtheorem{theorem}{Theorem}[section]
\newtheorem{lemma}[theorem]{Lemma}

\theoremstyle{definition}

\begin{document}
\title[An affirmative answer to Yamada's Conjecture]
{An affirmative answer to Yamada's Conjecture}

\author{Chuan Liu}
\address{(Chuan Liu): Department of Mathematics,
Ohio University Zanesville Campus, Zanesville, OH 43701, USA}
\email{liuc1@ohio.edu}
\author{Fucai Lin}
\address{(Fucai Lin): School of mathematics and statistics,
Minnan Normal University, Zhangzhou 363000, P. R. China}
\email{linfucai2008@aliyun.com; linfucai@mnnu.edu.cn}
\thanks{The second author is supported by the NSFC (No. 11571158), the Natural Science Foundation of Fujian Province (No. 2017J01405) of China, the Program for New Century Excellent Talents in Fujian Province University, the Project for Education Reform of Fujian Education Department (No. FBJG20170182), the Institute of Meteorological Big Data-Digital Fujian and Fujian Key Laboratory of Data Science and Statistics.}


\begin{abstract}
In this paper, we give an affirmative answer to Yamada's Conjecture
on free topological groups, which was posed in [K. Yamada, {\it
Fr\'echet-Urysohn spaces in free topological groups}, Proc. Amer.
Math. Soc., {\bf 130}(2002), 2461--2469.].
\end{abstract}

\maketitle

\section{Introduction and Preliminaries}
Throughout this paper, all topological spaces are assumed to be
Tychonoff, unless explicitly stated otherwise. Given a space $X$,
let $F(X)$ and $A(X)$ be the free topological group and free Abelian topological group over $X$ in the sense of
Markov respectively. For every $n\in\mathbb{N}$, let $F_{n}(X)$ and $A_{n}(X)$ denote the
subspaces of $F(X)$ and $A(X)$ respectively that consists of words of reduced length at most
$n$ with respect to the free basis $X$. A space $X$ is said to be
{\it Fr$\acute{e}$chet-Urysohn} if, for each $x\in
\overline{A}\subset X$, there exists a sequence $\{x_{n}\}$ such
that $\{x_{n}\}$ converges to $x$ and $\{x_{n}:
n\in\mathbb{N}\}\subset A$.

The free topological group $F(X)$ and the free abelian topological
group $A(X)$ over a Tychonoff space $X$ were introduced by Markov
\cite{MA1945} and intensively studied over the last half-century,
see for example \cite{AOP1989,LLL,LL, LLLL,U1991,Y1993,Y1997,Y1998,Y2016}.

\smallskip
In \cite{Y2002}, Yamada proved the following two theorems.

\begin{theorem}\cite{Y2002}
Let $X$ be a metrizable space. Then $F_5(X)$ is Fr\'echet-Urysohn if
and only if $X$ is discrete or compact.
\end{theorem}

\begin{theorem}\cite{Y2002}\label{t1}
Let $X$ be a metrizable space. Then $F_3(X)$ is Fr\'echet-Urysohn if
and only if the set of all the non-isolated points of $X$ is
compact.
\end{theorem}

Therefore, Yamada gave the following conjecture:

\smallskip
\noindent{\bf Yamada's Conjecture:} \cite{Y2002} If the set of all
non-isolated points of a metrizable space $X$ is compact, then
$F_4(X)$ is Fr\'echet-Urysohn.

\smallskip
In this paper, we shall give an affirmative answer to Yamada's
Conjecture.

Let $X$ be a non-empty Tychonoff space. Throughout this paper, $X^{-1}
:=\{x^{-1}: x\in X\}$, which is just a copy of
$X$. For every $n\in\mathbb{N}$, $F_{n}(X)$ denotes the
subspace of $F(X)$ that consists of all words of reduced length at
most $n$ with respect to the free basis $X$. Let $e$ be the neutral element of $F(X)$ (i.e., the empty
word). For every $n\in\mathbb{N}$ and an element
$(x_{1}, x_{2}, \cdots, x_{n})$ of $(X\bigoplus X^{-1}\bigoplus\{e\})^{n}$
we call $g=x_{1}x_{2}\cdots x_{n}$ a {\it word}. This word $g$ is
called {\it reduced} if it does not contain $e$ or any pair of
consecutive symbols of the form $xx^{-1}$ or $x^{-1}x$. It follows
that if the word $g$ is reduced and non-empty, then it is different
from the neutral element $e$ of $F(X)$. In particular, each element
$g\in F(X)$ distinct from the neutral element can be uniquely written
in the form $g=x_{1}^{r_{1}}x_{2}^{r_{2}}\cdots x_{n}^{r_{n}}$, where
$n\geq 1$, $r_{i}\in\mathbb{Z}\setminus\{0\}$, $x_{i}\in X$, and
$x_{i}\neq x_{i+1}$ for each $i=1, \cdots, n-1$, and the {\it support}
  of $g=x_{1}^{r_{1}} x_{2}^{r_{2}}\cdots x_{n}^{r_{n}}$ is defined as
  $\mbox{supp}(g) :=\{x_{1}, \cdots, x_{n}\}$. Given a subset $K$ of
  $F(X)$, we define $\mbox{supp}(K):=\bigcup_{g\in K}\mbox{supp}(g)$.

\section{The proof of Yamada's Conjecture}
Throughout this paper, we always assume that $(X, d)$ is a metric
space with a metric $d$ such that the set $K$ of all the
non-isolated points of $X$ is compact. For each $n\in\mathbb{N}$,
let $$V_n=\{x\in X: d(x, K)<1/n\}.$$  Let $W_1=X\setminus V_1$ and
$W_n=V_{n-1}\setminus V_n$ for each $n\geq 2$. It is easy to see
that each $W_n$ is a closed discrete subspace.

The following lemma play an important role in our proof.

\begin{lemma}\label{l1}
There is a compatible metric $\varrho$ on $(X, d)$, which satisfies
the following conditions:
\begin{enumerate}
\smallskip
\item $\varrho(x, y)=\varrho(y, x)$ for any $x, y\in X$;
\smallskip
\item $\varrho(x,y)=|i-j|/(i\cdot j)$ if $x\in W_i, y\in W_j, i\neq j$ and $d(x,
y)<|i-j|/(i\cdot j)$;
\smallskip
\item $\varrho(x, y)=1/(i\cdot (i+1))$ if $x, y\in W_i,
x\neq y$ and $d(x, y)<1/(i\cdot (i+1))$;
\smallskip
\item $\varrho(x, y)=d(x, y)$, otherwise.
\end{enumerate}
\end{lemma}

\begin{proof}
First we prove $\varrho$ is a metric on $X$. Clearly, it suffices to
prove the triangle inequality. Take arbitrary $x, y, z\in X$. We may
assume that $x, y, z$ are distinct each other.

\smallskip
{\bf Case 1} $|\{x, y, z\}\cap K|\geq 2$.

\smallskip
Then $\varrho(x, y)=d(x, y), \varrho(x, z)=d(x, z)$ and $\varrho(z,
y)=d(z, y)$. Since $d(x, y)\leq d(x, z)+d(z, y)$, $\varrho(x, y)\leq
\varrho(x, z)+\varrho(z, y)$.

\smallskip
{\bf Case 2} $|\{x, y, z\}\cap K|=1$.

\smallskip
Without loss of generality, we may assume that $x\in K$. Obviously,
there exist $i, j\in \mathbb{N}$ such that $y\in W_{i}$ and $z\in
W_{j}$. Then $\varrho(x, y)=d(x, y), \varrho(x, z)=d(x, z)$. It
follows from the definition of $\varrho$ that $d(y, z)\leq
\varrho(y, z)$. Therefore, $\varrho(x, y)=d(x, y)\leq d(x, z)+d(z,
y)=\varrho(x, z)+d(z, y)\leq\varrho(x, z)+\varrho(y, z)$;
$\varrho(x, z)=d(x, z)\leq d(x, y)+d(y, z)=\varrho(x, y)+d(y,
z)\leq\varrho(x, y)+\varrho(y, z)$. If $i=j$ and $d(y, z)<1/(i\cdot
(i+1))$, then $\varrho(y, z)=1/(i\cdot (i+1))$, hence $\varrho(y,
z)\leq\varrho(y, x)+\varrho(x, z)$ since $\varrho(x, y)\geq1/i$ and
$\varrho(x, z)\geq1/i$. If $i=j$ and $d(y, z)\geq1/(i\cdot (i+1))$,
then it is obvious. If $i\neq j$ and $d(y, z)<|i-j|/(i\cdot j)$, then
$\varrho(y, z)=|i-j|/(i\cdot j)$, hence $\varrho(y, z)\leq\varrho(y,
x)+\varrho(x, z)$ since $\varrho(x, y)\geq1/i$ and $\varrho(x, z)\geq1/j$.
If $i\neq j$ and $d(y, z)\geq|i-j|/(i\cdot j)$, then it is obvious.

\smallskip
{\bf Case 3} $|\{x, y, z\}\cap K|=0$.

\smallskip
Then there exist $i, j, k\in \mathbb{N}$ such that $x\in W_{i}, y\in
W_{j}, z\in W_{k}$. Without loss of generality, it suffices to prove
$\varrho(x, y)\leq\varrho(x, z)+\varrho(y, z)$.

If $i=j$ and $d(x, y)<1/(i\cdot (i+1))$, then $\varrho(x,
y)=1/(i\cdot (i+1))$. Moreover, $\varrho(x, z)\geq|i-k|/(i\cdot k)$ and $\varrho(y, z)\geq|i-k|/(i\cdot k)$. Assume that $i=k$, then $\varrho(x,
z)+\varrho(z, y)=1/(i\cdot (i+1))+1/(i\cdot (i+1))=2/(i\cdot (i+1))\geq 1/(i\cdot (i+1))=\varrho(x, y)$.
Hence it suffices to consider $i\neq k$. Then $\varrho(x, z)+\varrho(z, y)\geq|i-k|/(i\cdot k)+|i-k|/(i\cdot k)=2|i-k|/(i\cdot k)\geq 1/(i\cdot (i+1))=\varrho(x, y)$.
If $i=j$ and $d(x, y)\geq 1/(i\cdot (i+1))$, then $\varrho(x, y)=d(x, y)\leq d(x, z)+d(y, z)\leq\varrho(x, z)+\varrho(y, z)$. If
$i\neq j$ and $d(x, y)<|i-j|/(i\cdot j)$, then $\varrho(x,
y)=|i-j|/(i\cdot j)$, hence $\varrho(x, z)+\varrho(z,
y)\geq |i-k|/(i\cdot k)+|j-k|/(j\cdot k)=\frac{|ij-jk|+|ij-ik|}{ijk}\geq\frac{|ik-jk|}{ijk}=|i-j|/(i\cdot j)=\varrho(x, y)$. If $i\neq j$
and $d(x, y)\geq |i-j|/(i\cdot j)$, then $\varrho(x, y)=d(x, y)\leq d(x, z)+d(y, z)\leq\varrho(x, z)+\varrho(y, z)$.

Therefore, $\varrho$ is a metric. It easily check that the topology
generated by the metric $\varrho$ on $X$ is compatible with $(X,
d)$.
\end{proof}

\begin{lemma}\label{l2}
The metric $\varrho$ in Lemma~\ref{l1} has the following properties:

\smallskip
(1) For any $x\in X\setminus V_k, y\in X, y\neq x$, it has $\varrho(x,
y)>1/(k+1)^2$.

\smallskip
(2) If $\varrho(x_n, y_n)<1/n$ for each $n\in\mathbb{N}$, then there
exist sequences $\{x_{n_k}\}$ and $\{y_{n_k}\}$ such that
$x_{n_k}\to x_{0}$ and $y_{n_k}\to x_{0}$ as $k\rightarrow\infty$,
where $x_{n}, y_{n}\in X, x_{n}\neq y_{n}$ and $x_{0}\in K$.
\end{lemma}

Let $\varrho^*$ and $N_{\varrho}$ be defined in the proof in
\cite[Theorem 7.2.2]{AT2008}. For the convenience, we give out the
definitions.

Suppose that $e$ is the neutral element of the abstract free group
$F_{a}(X)$ on $X$. Extend $\varrho$ from $X$ to a metric
$\varrho^{\ast}$ on $X\cup\{e\}\cup X^{-1}$. Choose a point $x_{0}\in
X$ and for every $x\in X$, put
$$\varrho^*(e, x)=\varrho^*(e, x^{-1})=1+\varrho(x_{0}, x).$$ Then for $x, y\in X$, define the distance $\varrho^*(x^{-1}, y^{-1}), \varrho^*(x^{-1}, y)$ and $\varrho^*(x, y^{-1})$ by $$\varrho^*(x^{-1}, y^{-1})=\varrho^*(x, y)=\varrho(x, y),$$ $$\varrho^*(x^{-1}, y)=\varrho^*(x, y^{^{-1}})=\varrho^*(x, e)+\varrho^*(e, y).$$

Let $A$ be a subset of $\mathbb{N}$ such that $|A|=2n$ for some
$n\geq 1$. A {\it scheme} on $A$ is a partition of $A$ to pairs
$\{a_{i}, b_{i}\}$ with $a_{i}<b_{i}$ such that each two intervals
$[a_{i}, b_{i}]$ and $[a_{j}, b_{j}]$ in $\mathbb{N}$ are either
disjoint or one contains the other.

If $\mathcal{X}$ is a word in the alphabet $X\cup\{e\}\cup X^{-1}$,
then we denote the reduced form and the length of  $\mathcal{X}$ by
$[\mathcal{X}]$ and $\ell (\mathcal{X})$ respectively.

For each $n\in \mathbb{N}$, let $\mathcal{S}_{n}$ be the family of
all schemes $\varphi$ on $\{1, 2, \cdots, 2n\}$. As in
\cite{AT2008}, define
$$\Gamma_{\varrho}(\mathcal{X}, \varphi)=\frac{1}{2}\sum_{i=1}^{2n}\varrho^{\ast}(x_{i}^{-1}, x_{\varphi (i)}).$$
Then we define a prenorm $N_{\varrho}: F_{a}(X)\rightarrow [0,
+\infty)$ by setting $N_{\varrho}(g)=0$ if $g=e$ and
$$N_{\varrho}(g)=\inf\{\Gamma_{\varrho}(\mathcal{X}, \varphi):
[\mathcal{X}]=g, \ell (\mathcal{X})=2n, \varphi\in\mathcal{S}_{n},
n\in \mathbb{N}\}$$ if $g\in F_{a}(X)\setminus\{e\}$.

\begin{lemma}\label{l5}
Let
$$B=\{x^{\epsilon_1}y^{\epsilon_2}z^{\epsilon_3}t^{\epsilon_4}\in
F_{4}(X)\setminus F_{3}(X): x, y, z, t\in X, \epsilon_i\in\{-1, 1\},
\sum\epsilon_i=0, 1\leq i\leq 4\}.$$ If $e\in \overline{B}$, then
there is a convergent sequence
$\{h_n=x_n^{\epsilon_1}y_n^{\epsilon_2}z_n^{\epsilon_3}t_n^{\epsilon_4}\}$
in $B$ such that $h_n\to e$ as $n\rightarrow\infty$.
\end{lemma}

\begin{proof}
Let $\varrho^*$ and $N_{\varrho}$ be defined as above. It is known
that for each $n\in\mathbb{N}$,
$$U_{\varrho}(n)=\{g\in F_a(X): N_{\varrho}(g)<1/n\}\cap F_0(X)$$ is
an open neighborhood of $e$ in $F(X)$ by \cite[Theorem
7.2.2]{AT2008}. We divide the proof into the following two cases.

\smallskip
{\bf Case 1}: For each $n\in \mathbb{N}$, there exists
$x_n^{\epsilon_1(n)}y_n^{\epsilon_2(n)}z_n^{\epsilon_3(n)}t_n^{\epsilon_4(n)}\in
B\cap U_\varrho(n)$  such that $x_n, y_n, z_n, t_n\in V_n$, where
$V_n=\{x\in X: d(x, K)<1/n\}.$

\smallskip
Without loss of generality, we may assume that
$\epsilon_i(n)=\epsilon_i (i\leq 4)$. In fact, we can choose a
subsequence of
$\{x_n^{\epsilon_1(n)}y_n^{\epsilon_2(n)}z_n^{\epsilon_3(n)}t_n^{\epsilon_4(n)}:
n\in \mathbb{N}\}$. We can also assume that $x_n\to x, y_n\to y,
z_n\to z, t_n\to t$, where $x, y, z , t\in K$. Clearly,
$\sum\epsilon_i=0$. Then we have the following claim.

\smallskip
{\bf Claim 1:}
$x^{\epsilon_1}y^{\epsilon_2}z^{\epsilon_3}t^{\epsilon_4}=e$.

Fix an $n\in\mathbb{N}$, let $g_{n}=x_n^{\epsilon_1}y_n^{\epsilon_2}z_n^{\epsilon_3}t_n^{\epsilon_4}$ and let $D_{n}=\{x_n^{\epsilon_1}, y_n^{\epsilon_2}, z_n^{\epsilon_3}, t_n^{\epsilon_4}\}$. We claim that $$N_{\varrho}(g_{n})=\min\{\varrho^*(x_n^{-\epsilon_1},
y_n^{\epsilon_2})+\varrho^*(z_n^{-\epsilon_3},
t_n^{\epsilon_4}), \varrho^*(x_n^{-\epsilon_1},
t_n^{\epsilon_4})+\varrho^*(y_n^{-\epsilon_3},
z_n^{\epsilon_4})\}$$ if $N_{\varrho}(g_{n})<\frac{1}{n}$.

In fact, it follows from Claim 1 of the proof in \cite[Theorem 7.2.2]{AT2008} that there exist an almost reduced word $\Upsilon_{g_{n}}=p_{1}\cdots p_{2m}$  with $2m\leq 8$ and a scheme $\varphi_{g_{n}}$ such that $\Upsilon_{g_{n}}$ contains only the letters of $g_{n}$ or the letter $e$ and $N_{\varrho}(g_{n})=\Gamma_{\varrho}(\Upsilon_{g_{n}}, \varphi_{g_{n}})$, where $\Gamma_{\varrho}(\Upsilon_{g_{n}}, \varphi_{g_{n}})=\frac{1}{2}\sum_{i=1}^{2m}\varrho^{\ast}(p_{i}^{-1}, p_{\varphi_{g_{n}}(i)})$. If $p_{i}\in D_{n}$, then $p_{\varphi_{g_{n}}(i)}\in D_{n}$; otherwise, $p_{\varphi_{g_{n}}(i)}=e$, then $\Gamma_{\varrho}(\Upsilon_{g_{n}}, \varphi_{g_{n}})\geq 1$, which is a contradiction with $N_{\varrho}(g_{n})<\frac{1}{n}$.

\smallskip
{\bf Subcase 1.1:} $\varrho^*(x_n^{-\epsilon_1},
y_n^{\epsilon_2})+\varrho^*(z_n^{-\epsilon_3},
t_n^{\epsilon_4})<1/n$ for infinitely many $n\in\mathbb{N}$.

\smallskip
Then $\epsilon_1=-\epsilon_2$ and $\epsilon_3=-\epsilon_4$. If
$x\neq y$, then $\varrho(x, y)=r>0$. It is easy to see that there is
$k\in \mathbb{N}$ such that $\varrho(x, x_i)<r/3$, $\varrho(x_i,
y_i)<r/3$, $\varrho(y_i, y)<r/3$ whenever $i>k$. Then $r=\varrho(x,
y)\leq \varrho(x, x_i)+\varrho(x_i, y_i)+\varrho(y_i,
y)<r/3+r/3+r/3=r$. This is a contradiction, hence $x=y$. Similarly,
$z=t$. Therefore
$x^{\epsilon_1}y^{\epsilon_2}z^{\epsilon_3}t^{\epsilon_4}=e$.

\smallskip
{\bf Subcase 1.2:} $\varrho^*(x_n^{-\epsilon_1},
t_n^{\epsilon_4})+\varrho^*(y_n^{-\epsilon_3},
z_n^{\epsilon_4})<1/n$ for infinitely many $n\in\mathbb{N}$.

\smallskip
Then $\epsilon_1=-\epsilon_4, \epsilon_2=-\epsilon_3$. By the proof
of Subcase 1.1, we can prove that $x=t, y=z$. Then
$x^{\epsilon_1}y^{\epsilon_2}z^{\epsilon_3}t^{\epsilon_4}=e$. The proof of Claim 1 is completed.

By Claim 1, we see that Lemma~\ref{l5} holds.

\smallskip
{\bf Case 2:} There is $n\in \mathbb{N}$ such that for any
$a^{\epsilon_1}b^{\epsilon_2}c^{\epsilon_3}d^{\epsilon_4}\in
U_\varrho(n)\cap B=A$, one of $a, b, c, d$ is not in $V_n$.

\smallskip
We only consider the case $$e\in \overline{\{a^{\epsilon_1}b^{\epsilon_2}c^{\epsilon_3}d^{\epsilon_4}\in A: a\not\in V_{n}\}\cup\{a^{\epsilon_1}b^{\epsilon_2}c^{\epsilon_3}d^{\epsilon_4}\in A: b\not\in V_{n}\}},$$otherwise consider the set
$(U_\varrho(n)\cap B)^{-1}$.
First, we prove that $e\in \overline{\{a^{\epsilon_1}b^{\epsilon_2}c^{\epsilon_3}d^{\epsilon_4}\in A: a\not\in V_{n}\}}$, then it suffices to prove $$e\not\in\overline{\{a^{\epsilon_1}b^{\epsilon_2}c^{\epsilon_3}d^{\epsilon_4}\in A: b\not\in V_{n}\}}.$$ Let $$A_{1}=\{a^{\epsilon_1}b^{\epsilon_2}c^{\epsilon_3}d^{\epsilon_4}\in A: b\not\in V_{n}\}.$$
Assume that $e\in\overline{A_{1}}.$ In order to obtain a contradiction, we find a neighborhood $W_1$ of $e$ such that $W_{1}\cap A_{1}=\emptyset$. Indeed,
let $$W_1=\{g\in F(X): N_\varrho(g)<1/(n+1)^2\}.$$ Then $W_1$ is an
open neighborhood of $e$. Moreover, for any
$g=a^{\epsilon_1}b^{\epsilon_2}c^{\epsilon_3}d^{\epsilon_4}\in A_{1}$,
we prove that $N_\varrho(g)\geq1/(n+1)^2$.
Indeed, it follows from Claim 1 of the proof of \cite[Theorem 7.2.2]{AT2008} that there exist an almost reduced word $\Upsilon_{g}=x_{1}\cdots x_{2m}$ and a scheme $\varphi_{g}$ such that  satisfies the following conditions:

\smallskip
(i) $\Upsilon_{g}$ contains only the letters of $g$ or the letter $e$;

\smallskip
(ii) $[\Upsilon_{g}]=g$ and $\ell(\Upsilon_{g})\leq 2\ell(g)$;

\smallskip
(iii) $N_{\varrho}(g)=\Gamma_{\varrho}(\Upsilon_{g}, \varphi_{g})$.

\smallskip
We claim that $\Gamma_{\varrho}(\Upsilon_{g}, \varphi_{g})\geq\frac{1}{(n+1)^{2}}$. Indeed, it follows from (i) and (ii) that there exists $i_{0}\leq 2m$ such that $x_{i_{0}}=b^{\epsilon_2}$, then $\varrho^*(x_{i_{0}}^{-1}, x_{\varphi_{g}(i_{0})})=\varrho^*(b^{-\epsilon_2}, x_{\varphi_{g}(i_{0})})$. Then we can complete the proof by the following (a)-(d).

\smallskip
(a) If $x_{\varphi_{g}(i_{0})}=e$, then $\varrho^*(b^{-\epsilon_2}, x_{\varphi_{g}(i_{0})})\geq 1$.

\smallskip
(b) If $x_{\varphi_{g}(i_{0})}=a^{\epsilon_1}$, then it follows from (1) of Lemma~\ref{l2} that $\varrho^*(b^{-\epsilon_2}, x_{\varphi_{g}(i_{0})})=\varrho^*(b^{-\epsilon_2}, a^{\epsilon_1})>\frac{1}{(n+1)^{2}}$ since $a^{\epsilon_1}b^{\epsilon_2}\neq e$ and $b\not\in V_{n}$.

\smallskip
(c) Assume $x_{\varphi_{g}(i_{0})}=c^{\epsilon_3}$. If $\epsilon_2=\epsilon_3$ and $b=c$, then $x_{\varphi_{g}(i_{0})}=b^{\epsilon_2}$, hence $\varrho^*(b^{-\epsilon_2}, x_{\varphi_{g}(i_{0})})=\varrho^*(b^{-\epsilon_2}, b^{\epsilon_2})\geq 1\geq\frac{1}{(n+1)^{2}}$. If $b\neq c$, then it follows from (1) of Lemma~\ref{l2} that $\varrho^*(b^{-\epsilon_2}, x_{\varphi_{g}(i_{0})})=\varrho^*(b^{-\epsilon_2}, c^{\epsilon_3})>\frac{1}{(n+1)^{2}}$ since $b\neq c$ and $b\not\in V_{n}$.

\smallskip
(d) Assume $x_{\varphi_{g}(i_{0})}=d^{\epsilon_4}$. Obviously, there exists a $j_{0}\leq 2m$ such that $x_{j_{0}}=c^{\epsilon_3}$. Since $\varphi_{g}$ is a scheme, it has $x_{\varphi_{g}(j_{0})}=e$, hence $\varrho^*(x_{j_{0}}^{-1}, x_{\varphi_{g}(j_{0})})\geq 1$. Then $$\Gamma_{\varrho}(\Upsilon_{g}, \varphi_{g})\geq\varrho^*(b^{-\epsilon_2}, x_{\varphi_{g}(i_{0})})+\varrho^*(x_{j_{0}}^{-1}, x_{\varphi_{g}(j_{0})})\geq 1\geq\frac{1}{(n+1)^{2}}.$$

Therefore, $W_{1}\cap A_{1}=\emptyset$, which is a contradiction.
Hence the point $e$ belongs to the closure of
$\{a^{\epsilon_1}b^{\epsilon_2}c^{\epsilon_3}d^{\epsilon_4}\in A:
a\not\in V_{n}\}$. Further, we claim that $e$ does not belong to the closure of the set
$$A_2=\{a^{\epsilon_1}b^{\epsilon_2}c^{\epsilon_3}d^{\epsilon_4}\in
A: a\neq d, a\not\in V_{n}\}\cup\{a^{\epsilon_1}b^{\epsilon_2}c^{\epsilon_3}d^{\epsilon_4}\in
A: a=d, \epsilon_1=\epsilon_4, a\not\in V_{n}\}.$$

Suppose not, assume that $e\in \overline{A_2}$. In order to obtain a contradiction, it suffices to prove that $W_1\cap A_2=\emptyset$ by a similar proof above.

\smallskip
Therefore, $e$ belongs to the closure of
$A_3=\{a^{\epsilon_1}b^{\epsilon_2}c^{-\epsilon_2}a^{-\epsilon_1}\in
A: a\not\in V_{n}\}$. Next we prove that there exists a convergent sequence in $A_3$ which converges to $e$.

\smallskip
Let $D=\{a\in X:
a^{\epsilon_1}b^{\epsilon_2}c^{-\epsilon_2}a^{-\epsilon_1}\in
A_3\}$, and let
$A_a=\{a^{\epsilon_1}b^{\epsilon_2}c^{-\epsilon_2}a^{-\epsilon_1}\in
A_3\}$ for each $a\in D$. It is obvious that $A_3=\bigcup_{a\in D}
A_a$. We claim that there exist $a\in D$, an infinite subset $M$ of $\mathbb{N}$, $b_i\in V_i$ and $c_i\in V_i$ for any
$i\in M$ such that
$$a^{\epsilon_1}b_i^{\epsilon_2}c_i^{-\epsilon_2}a^{-\epsilon_1}\in
A_3\cap U_{\varrho}(i).$$

Suppose not, for each $a\in D$, there is $n_a\in \mathbb{N}$ $(n_\alpha>n)$
such that either $b\notin V_{n_a}$ or $c\notin V_{n_a}$ for each
$a^{\epsilon_1}b^{\epsilon_2}c^{-\epsilon_2}a^{-\epsilon_1}\in
A_{a}$. Without loss of generality, we may assume that $e$ belongs to the closure of
$A_{4}=\{a^{\epsilon_1}b^{\epsilon_2}c^{-\epsilon_2}a^{-\epsilon_1}\in
A_{3}: a\in D, b\not\in V_{n_{a}}, c\in V_{n}\}$.
Let $C^{\prime}=\{b, c\in X: a^{\epsilon_1}b^{\epsilon_2}c^{-\epsilon_2}a^{-\epsilon_1}\in A_{4}\}$ and $C=\overline{C^{\prime}}$. Obviously, $C\cap D=\emptyset$, $D$ is closed discrete and $C$ is closed in $X$. Define a mapping $f=\pi\circ\overline{\psi}: F(C\oplus D)\rightarrow A(C\times F(D))$, where the mappings $\pi$ and $\overline{\psi}$ are defined in the proof of \cite[Theorem 2.4]{Y2002}. By our definition of $C$ and $D$, it is easy to see that $\pi$, $\overline{\psi}$ and $f$ are all continuous. Therefore, $0=f(e)\in \overline{f(A_{4})}$. For each $a^{\epsilon_1}b^{\epsilon_2}c^{-\epsilon_2}a^{-\epsilon_1}\in A_{4}$, we have
\begin{eqnarray}
\overline{\psi}(a^{\epsilon_1}b^{\epsilon_2}c^{-\epsilon_2}a^{-\epsilon_1})&=&\overline{\psi}(a^{\epsilon_1})\overline{\psi}(b^{\epsilon_2})\overline{\psi}(c^{-\epsilon_2})\overline{\psi}(a^{-\epsilon_1})\nonumber\\
&=&\psi(a)^{\epsilon_1}\psi(b)^{\epsilon_2}\psi(c)^{-\epsilon_2}\psi(a)^{-\epsilon_1}\nonumber\\
&=&(a^{\epsilon_1}, 0)(e, \epsilon_2(b, e))(e, -\epsilon_2(c, e))(a^{-\epsilon_1}, 0)\nonumber\\
&=& (a^{\epsilon_1}, \epsilon_2(b, a^{\epsilon_1}))(a^{-\epsilon_1}, -\epsilon_2(c, e))\nonumber\\
&=&(e, \epsilon_2(b, a^{\epsilon_1})-\epsilon_2(c, a^{\epsilon_1})),\nonumber
\end{eqnarray}
hence $f(a^{\epsilon_1}b^{\epsilon_2}c^{-\epsilon_2}a^{-\epsilon_1})=\pi\circ \overline{\psi}(a^{\epsilon_1}b^{\epsilon_2}c^{-\epsilon_2}a^{-\epsilon_1})=\epsilon_2(b, a^{\epsilon_1})-\epsilon_2(c, a^{\epsilon_1})\in A_{2}(C\times F(D))$. By the arbitrary, it follows that $\overline{f(A_{4})}\subset A_{2}(C\times F(D))$. Since $C\times F(D)$ is a metrizable space, it follows from \cite[Proposition 4.8]{Y1998} that $A_{2}(C\times F(D))$ is a Fr\'{e}chet-Urysohn space. Then there exists a sequence $\mathcal{S}=\{\epsilon_2(k)(b_{k}, a_{k}^{\epsilon_1(k)})-\epsilon_2(k)(c_{k}, a_{k}^{\epsilon_1(k)}): k\in\mathbb{N}\}$ which converges to $0$, where $\epsilon_1(k), \epsilon_2(k)\in\{1, -1\}$ for each $k\in\mathbb{N}$. Then $F=\overline{\mbox{supp}(\mathcal{S}\cup\{0\})}$ is a compact subset in $C\times F(D)$. Since the projective mappings $\pi_{1}$ and $\pi_{2}$ are continuous from $C\times F(D)$ to $C$ and $F(D)$ respectively, the sets $\pi_{1}(F)$ and $\pi_{2}(F)$ are compact in $C$ and $F(D)$ respectively. Then the set $F_{1}=\{a_{k}^{\epsilon_1(k)}: k\in\mathbb{N}\}$ is a finite set since $F(D)$ is discrete. Therefore, the set $F_{2}=\{b_{k}: k\in\mathbb{N}\}$ is also finite since $\{b_{k}: k\in\mathbb{N}\}\cap V_{n_{0}}=\emptyset$, where $n_{0}=\max\{n_{a}: a\in F_{1}\}$. Therefore, there exist $a\in F_{1}$ and $b\in F_{2}$ such that some subsequence $\{\epsilon_2(n_{k})(b, a)-\epsilon_2(n_{k})(c_{n_{k}}, a): k\in\mathbb{N}\}$ of $\mathcal{S}$ converges to $0$. Without loss of generality, we may assume that $(b, a)-(c_{n_{k}}, a)\rightarrow 0$ as $k\rightarrow\infty$, that is, $(c_{n_{k}}, a)\rightarrow (b, a)$ as $k\rightarrow\infty$. Then $c_{n_{k}}=b$ for any sufficiently large $k$, which is a contradiction.

\smallskip
Then $a^{\epsilon_1}b_i^{\epsilon_2}c_i^{-\epsilon_2}a^{-\epsilon_1}\to
a^{\epsilon_1}b^{\epsilon_2}c^{-\epsilon_2}a^{-\epsilon_1}$ as
$i\rightarrow\infty$, where $b, c\in K$. In viewing of the proof of
Subcase 1.1, we could prove that $b=c$, hence
$a^{\epsilon_1}b_i^{\epsilon_2}c_i^{-\epsilon_2}a^{-\epsilon_1}\to
a^{\epsilon_1}b^{\epsilon_2}b^{-\epsilon_2}a^{-\epsilon_1}=e$ as
$i\rightarrow\infty$.
\end{proof}

For each $g=a_{1}\ldots a_{n}\in F_{n}(X)\setminus F_{n-1}(X)$ with
$a_{1}, \ldots, a_{n}\in X\cup X^{-1}$, denote by $\mathcal{P}_X(g)$
the subfamily of $P(X)$ consisting of all $\varrho$ such that
$\varrho^*(a_{i}^{-1}, a_{i+1})\geq 1$ for each $i<n$. For each
$\varrho\in\mathcal{P}_X(g)$, put $U_{\varrho}(g)=\{x_{1}\ldots
x_{i}y^{\epsilon}z^{-\epsilon}x_{i+1}\ldots x_{n}: x_{1}, \ldots,
x_{n}\in X\cup X^{-1}, y, z\in X, \epsilon=\pm 1, 0\leq i\leq n,\
\mbox{and}\ \varrho(y, z)+\sum_{k=1}^{n}\varrho(a_{k}, x_{k})<1\}$.

\begin{lemma}\label{th}
\cite[Theorem 7.2.11]{AT2008}. The family $\{U_{\varrho}(g):
\varrho\in \mathcal{P}_X(g)\}$ is an open base for $F_{n+2}(X)$ at
any point $g\in F_n(X)\setminus F_{n-1}(X)$.
\end{lemma}

\begin{lemma}\label{l6}
Let $B\subset F_{4}(X)\setminus F_{3}(X).$ If $g=a_1a_2\in
\overline{B}$ is a reduced form, where $a_1, a_2\in X\cup X^{-1}$, then there is a
convergent sequence
$\{h_n=x_n^{\epsilon_1}y_n^{\epsilon_2}z_n^{\epsilon_3}t_n^{\epsilon_4}: h_{n}\in B\}$
such that $h_n\to g$ as $n\rightarrow\infty$.
\end{lemma}

\begin{proof}
Let $\varrho$ be the metric on $X$ in Lemma~\ref{l1}. If $\varrho^*(a_1^{-1}, a_2)\geq 1$, then $\varrho\in
\mathcal{P}_X(g)$; if $\varrho^*(a_1^{-1}, a_2)=r<1$, then
$(1/r)\varrho\in \mathcal{P}_X(g)$. We may find a $k\in \mathbb{N}$
such that $k>1/r$, if $r\geq 1$, then $k=1$.

Let $\varrho_n=nk\cdot \varrho$ for each $n\in \mathbb{N}$. By Lemma
~\ref{th}, $U_{\varrho_n}(g)\cap B\neq \emptyset$ for each
$n\in\mathbb{N}$, thus without loss of generality it can take an
arbitrary point $x_1(n)y^{\epsilon}(n)z^{-\epsilon}(n)x_2(n)\in
U_{\varrho_n}(g)\cap B$. Then $n\cdot \varrho(y(n), z(n))<1$,
$n\cdot \varrho(x_1(n), a_1)<1$ and $n\cdot \varrho(x_2(n), a_2)<1$.
By Lemma~\ref{l2}, $y(n)\to f, z(n)\to f$, $x_i(n)\to a_i (i\leq 2)$
as $n\rightarrow\infty$. Therefore,
$x_1(n)y^{\epsilon}(n)z^{-\epsilon}(n)x_2(n)\to
a_1ff^{-1}a_2=a_1a_2=g$.
\end{proof}

Now, we can give an affirmative answer to Yamada's
Conjecture\footnote{Professor K. Yamada informed us that he just
proved this conjecture independently.}.

\begin{theorem}
Let $X$ be a metrizable space in which the set of all the
non-isolated points is compact. Then $F_4(X)$ is Fr\'echet-Urysohn.
\end{theorem}

\begin{proof}
Assume that $g\in \overline{B}$, where $B\subset F_4(X)$ and $g\in
F_4(X)$. Next we prove that there exists a sequence in $B$ which
converges to $g$. It is well-known that $$F_4(X)=[(F_{4}(X)\setminus
F_{3}(X))\cup (F_{2}(X)\setminus F_{1}(X))\cup\{e\}]\oplus
[(F_{3}(X)\setminus F_{2}(X))\cup (F_{1}(X)\setminus \{e\})].$$

If $g\in \overline{F_3(X)\cap B}$, we can find a sequence in $B$
converging to $g$ since $F_3(X)$ is Fr\'echet-Urysohn by
Theorem~\ref{t1}.

If $g\in \overline{(F_4(X)\setminus F_3(X))\cap B}$, we consider the
following cases:

(a) $g\in F_4(X)\setminus F_3(X)$. By \cite[Theorem 7.6.2]{AT2008},
$F_{4}(X)\setminus F_{3}(X)$ is metrizable, we can find a sequence
in $B$ converges to $g$.

(b) $g\in F_2(X)\setminus F_1(X)$. By Lemma ~\ref{l6},we can find a
sequence in $B$ converges to $g$.

(c) $g=e$. By Lemma ~\ref{l5}, there is a sequence in $B$ converging
to $e$.

Therefore $F_4(X)$ is Fr\'echet-Urysohn.

\end{proof}

{\bf Acknowledgements}. The authors are thankful to the
referee for valuable remarks and corrections and all other sort of help related to the content of this article.


\begin{thebibliography}{99}
\bibitem{AOP1989} A.V. Arhangel'ski\v\i, O.G. Okunev, V.G. Pestov,  {\it
  Free topological groups over metrizable spaces}, Topology Appl., {\bf 33}(1989),
  63--76.

\bibitem{AT2008} A. Arhangel'ski\v{\i}, M. Tkachenko,
  Topological groups and related structures, Atlantis Press, Paris; World
  Scientific Publishing Co. Pte. Ltd., Hackensack, NJ, 2008.

    \bibitem{LLL} Z. Li, F. Lin, C. Liu,  {\it Networks on free topological
  groups}, Topology Appl., {\bf 180} (2015), 186--198.

  \bibitem{LL} F. Lin, C. Liu,  {\it $S_{\omega}$ and $S_{2}$ on free
  topological groups}, Topology Appl., {\bf 176}(2014), 10--21.

  \bibitem{LLLL} F. Lin, C. Liu,  {\it The $k$-spaces property of the free Abelian topological groups over non-metrizable La\v{s}nev spaces}, Topology Appl., {\bf 220}(2017), 31--42.

  \bibitem{MA1945} A.A. Markov, {\it On free topological groups}, Amer. Math.
  Soc. Translation, {\bf 8} (1962), 195--272.

\bibitem{U1991} V. V. Uspenskii, Free topological groups of
metrizable spaces, Math. USSR Izv. 37(1991)657-680.

\bibitem{Y1993} K. Yamada, {\it Characterizations of a metrizable space such that every $A_n(X)$ is a $k$-space}, Topology Appl., {\bf 49}(1993), 74--94.

  \bibitem{Y1997} K. Yamada, {\it Tightness of free abelian topological groups and of
finite products of sequential fans}, Topology Proc., {\bf 22}(1997),
363--381.

  \bibitem{Y1998} K. Yamada, {\it Metrizable subspaces of free topological groups
  on metrizable spaces}, Topology Proc., {\bf 23}(1998), 379--409.

\bibitem{Y2002} K. Yamada, {\it Fr\'echet-Urysohn spaces in free topological groups}, Proc. Amer. Math. Soc., {\bf 130}(2002), 2461--2469.

\bibitem{Y2016} K. Yamada, {\it Fr\'echet-Urysohn subspaces of free topological groups}, Topology Appl., {\bf 210}(2016), 81--89.
\end{thebibliography}
\end{document}